\documentclass[10pt,reqno, english]{amsart}  
\usepackage[utf8]{inputenc}
\usepackage[T1]{fontenc}
\usepackage{amsmath,amsthm}
\usepackage{amsfonts,amssymb,enumerate}
\usepackage{url}
\usepackage{mathtools}  
\usepackage[colorlinks=true,urlcolor=blue,linkcolor=red,citecolor=magenta]{hyperref}
\usepackage{enumerate, paralist}
\usepackage{anysize}

\usepackage{units}

\theoremstyle{plain}
\newtheorem{theorem}{Theorem}[section]
\newtheorem{lemma}[theorem]{Lemma}

\newtheorem{conjecture}[theorem]{Conjecture}

\theoremstyle{definition}

\newtheorem{remark}[theorem]{Remark}

\newcommand{\R}{\mathbb{R}}
\newcommand{\Z}{\mathbb{Z}}
\newcommand{\conv}{\mathrm{conv}}
\newcommand{\dist}{\mathrm{dist}}
\newcommand{\Ind}{\mathrm{I}}
\newcommand{\KG}{\mathrm{KG}}

\begin{document}

\title[Fair splittings by independent sets]{Fair splittings by independent sets in sparse graphs}



\author[A.~Black \and U.~Cetin \and F.~Frick \and A.~Pacun \and L.~Setiabrata]{Alexander Black \and Umur Cetin \and Florian Frick \and Alexander Pacun \and Linus Setiabrata}

\address[AB, LS]{Dept.\ Math., Cornell University, Ithaca, NY 14853, USA}
\email{\{ab2776, ls823\}@cornell.edu} 

\address[UC]{Dept.\ Math., Bilkent University, 06800 Ankara, Turkey}
\email{sabriumurcetin@gmail.com} 

\address[FF]{Dept.\ Math.\ Sciences, Carnegie Mellon University, Pittsburgh, PA 15213, USA}
\email{frick@cmu.edu} 

\address[AP]{Math.\ Dept., Stony Brook University, Stony Brook, NY 11794, USA}
\email{alexander.pacun@stonybrook.edu}


\begin{abstract}
\small
Given a partition $V_1 \sqcup V_2 \sqcup \dots \sqcup V_m$ of the vertex set of a graph, we are interested in finding multiple disjoint independent sets that contain the correct fraction of vertices of each~$V_j$. We give conditions for the existence of $q$ such independent sets in terms of the topology of the independence complex. We relate this question to the existence of $q$-fold points of coincidence for any continuous map from the independence complex to Euclidean space of a certain dimension, and to the existence of equivariant maps from the $q$-fold deleted join of the independence complex to a certain representation sphere of the symmetric group. As a corollary we derive the existence of $q$ pairwise disjoint independent sets accurately representing the $V_j$ in certain sparse graphs for $q$ a power of a prime.
\end{abstract}

\date{September 10, 2018}
\maketitle

\section{Introduction}

Given a graph $G$ whose vertex set is partitioned into $V_1 \sqcup V_2 \sqcup \dots \sqcup V_m$ and $\varepsilon > 0$, we say that a set of vertices $S$ is a \emph{fair $\varepsilon$-representation} if $|S \cap V_j| \ge \lfloor \varepsilon \cdot |V_j| \rfloor$ for all~$j$ and that $S$ is an \emph{almost fair $\varepsilon$-representation} if $|S \cap V_j| \ge \lfloor \varepsilon\cdot |V_j| \rfloor - 1$ for all~$j$. Recently, Aharoni, Alon, Berger, Chudnovsky, Kotlar, Loebl, and Ziv~\cite{aharoni2017} studied the existence of (almost) fair representations by independent sets, that is, sets that do not contain both endpoints of any edge, for $\varepsilon = \frac{1}{\chi(G)}$. For example, they showed that for any partition of the vertex set of an $n$-cycle an almost fair $\frac12$-representation by an independent set always exists (even for odd~$n$).

Alishahi and Meunier~\cite{alishahi2017} showed that for any partition $V_1 \sqcup V_2 \sqcup \dots \sqcup V_m$ of the vertex set of a path there are two disjoint independent sets $S_1$ and $S_2$ that both are an almost fair $\frac12$-representation. That is, most of the vertex set of the path is split into two independent sets fairly representing the partition. Here we initiate more generally the study of (almost) fair splittings by independent sets for arbitrary graphs. Given a partition $V_1 \sqcup V_2 \sqcup \dots \sqcup V_m$ of the vertex set of a graph, we say that pairwise disjoint subsets $S_1, \dots, S_q \subset V$ are a \emph{fair splitting} of the partition if $|S_i \cap V_j| \ge \lfloor\frac{|V_j|}{q}\rfloor$ for all~$i$ and~$j$ and an \emph{almost fair splitting} if $|S_i \cap V_j| \ge \lfloor\frac{|V_j|+1}{q}\rfloor-1$ for all~$i$ and~$j$ and $|V_j \setminus \bigcup_i S_i| \le q-1$ for all~$j$. 

By induction Alishahi and Meunier can extend their almost fair splitting result for $q=2$ and a path. They are able to show that for the graph $G$ on vertex set $\{1,2,\dots,n\}$ with an edge $(i,j)$ if and only if $|i-j| < q$ and $i \ne j$, there is an almost fair splitting by $q$ independent sets provided that $q$ is a power of two. They actually show more, namely that the sets $S_i$ differ in cardinality by at most one. We will refer to this as a \emph{balanced} almost fair splitting. Balanced fair splittings are defined similarly.

We study the problem of fair splittings by independent sets in general sparse graphs and beyond parameters that are powers of two. We give two sufficient conditions for the existence of almost fair splittings by independent sets in terms of the topology of the independence complex~$\Ind(G)$ of~$G$, that is, the simplicial complex of all independent sets in~$G$. There is an almost fair splitting if $G$ is sufficiently sparse--and thus $\Ind(G)$ sufficiently dense--that any continuous map from $\Ind(G)$ into Euclidean space of an appropriate dimension exhibits a $q$-fold point. More precisely:

\begin{theorem}
\label{thm:main}
	Let $G$ be a graph on vertex set~$V$, and let $V_1 \sqcup V_2 \sqcup \dots \sqcup V_m$ be a partition of~$V$. Suppose the $V_j$ have cardinalities such that there are integers $q \ge 2$, $n \ge m+1$, and $k_1, \dots, k_m \ge 1$ with $|V_j| = qk_j-1$ and $|V| = (q-1)n+1$. Furthermore, suppose that for any continuous map $F\colon \Ind(G) \longrightarrow \R^{n-1}$ there are $q$ pairwise disjoint faces $S_1, \dots, S_q$ of $\Ind(G)$ with $F(S_1) \cap \dots \cap F(S_q) \ne \emptyset$. Then $G$ admits a balanced almost fair splitting by $q$ independent sets.
\end{theorem}

This result is a consequence of the constraint method of Blagojevi\'c, the third author, and Ziegler~\cite{blagojevic2014}. While in certain concrete situations it might be possible to reduce to the case that $|V_j|+1$ is divisible by $q$ for all~$j$, this strict requirement on the sizes of the sets $V_j$ is undesirable in practice. The third author extended the methods of~\cite{blagojevic2014} to study topological generalizations of certain results in geometric Ramsey theory~\cite{frick2017-2}. These methods can be employed to establish a sufficient criterion for the existence of fair splittings for sets $V_j$ of arbitrary size in terms of the nonexistence of equivariant maps from the $q$-fold deleted join of the independence complex to a certain representation sphere of the symmetric group; see Section~\ref{sec:prel} for notations and definitions. We develop the following configuration space -- test map scheme:

\begin{theorem}
\label{thm:main-eq}
	Let $G$ be a graph on vertex set~$V$, and let $V_1 \sqcup V_2 \sqcup \dots \sqcup V_m$ be a partition of~$V$. Let $n \ge m+1$ and $q \ge 2$ be integers such that $|V_j| \ge q-1$ for all~$j$, $|V| \le (q-1)n+1$, and such that there is no $\mathfrak S_q$-equivariant map $\Ind(G)^{*q}_\Delta \longrightarrow S(W_q^{\oplus n})$. Then $G$ admits an almost fair splitting by $q$ independent sets.
\end{theorem}

This is a proper strengthening of Theorem~\ref{thm:main}; see Remark~\ref{rem:comparison}. We generalize this result by proving that one can add a Hamiltonian path to~$G$ and still find an almost fair splitting by $q$ independent sets in this augmented graph; see Theorem~\ref{thm:main+path}. This follows from combining the equivariant-topological approach of Theorem~\ref{thm:main-eq} with ascertaining the vanishing of certain obstructions via convex geometry. We derive the following consequences:

\begin{compactenum}[(1)]
	\item\label{it:sparse} 
	Let $q$ be a prime power. For any partition $V_1 \sqcup V_2 \sqcup \dots \sqcup V_m$ of the vertex set of a graph $G$ into sufficiently large sets $V_j$ with $2N(v) + N^2(v) < q$ for every vertex~$v$, there is an almost fair splitting by $q$ independent sets. Here $N(v)$ denotes the number of neighbors of~$v$, and $N^2(v)$ denotes the number of vertices at distance precisely two from~$v$.
	\item\label{it:optimal} 
	Let $q$ be a prime. Let $G$ be the edge-disjoint union of a path of length $(q-1)n+1$ and vertex-disjoint cliques of size $q-1$.
	Then there is an almost fair splitting by $q$ independent sets. 
	\item\label{it:several-transversals}
	If the vertex set of a graph $G$ is partitioned into $m$ sufficiently large sets, then there are multiple independent sets that each intersect each part and cover more than half of the vertex set of~$G$; see Theorem~\ref{thm:several-transversals} for a precise statement.
\end{compactenum}

We discuss preliminaries in Section~\ref{sec:prel} and present a first idea of our reasoning for a simple example in Section~\ref{sec:cycle}. Section~\ref{sec:main} contains proofs of Theorem~\ref{thm:main} and Theorem~\ref{thm:main-eq}. In Section~\ref{sec:consequences} we deduce Items~(\ref{it:sparse}), (\ref{it:optimal}), and~(\ref{it:several-transversals}) as consequences of the main result; see Theorems~\ref{thm:corollaries} and~\ref{thm:several-transversals}. Section~\ref{sec:path} treats almost fair splittings of a path with additional restrictions on the independent sets; in particular, we prove approximations to a conjecture of Alishahi and Meunier. Section~\ref{sec:kneser} explains a connection between fair splittings by independent sets and chromatic numbers of Kneser hypergraphs. In particular, this leads to a new proof of Alishahi and Meunier's almost fair splitting result. While the focus of our work is on independence complexes of graphs, our methods can establish results for almost fair splittings of any simplicial complex. We remark on these generalizations in Section~\ref{sec:rem}.

\section{Preliminaries}
\label{sec:prel}

Here we collect the relevant definitions and notations. We refer to Matou\v sek's book~\cite{matousek2008} for more details.

\medskip

\noindent
\textbf{Simplicial complexes.} A simplicial complex $K$ is a set of sets such that $\sigma \in K$ and $\tau \subset \sigma$ implies $\tau \in K$. All simplicial complexes considered in this manuscript will be finite. Any set in $K$ is called a face of~$K$. The set of singleton sets (the minimal nonempty faces) is referred to as the vertex set of~$K$. If $V$ is the vertex set of $K$ then $K$ admits a geometric realization in~$\R^V$: it consists of the convex hulls of the form $\conv\{e_i \: : \: i \in \sigma\}$, where $\sigma$ ranges over the faces of~$K$. We will think of simplicial complexes as topological spaces in this way. If $L \subset K$ is a simplicial complex as well, we refer to it as subcomplex of~$K$. For a face $\sigma \in K$ define its dimension $\dim \sigma$ by $|\sigma|-1$, and the dimension of~$K$, denoted $\dim K$, by the largest dimension of a face of~$K$. For an integer $k\ge 0$ the $k$-skeleton $K^{(k)}$ of $K$ is the set of all faces of~$K$ of dimension at most~$k$. For a simplicial complex $K$ we denote the barycentric subdivision by~$K'$; it is the simplicial complex whose vertex set is the set of nonempty faces of~$K$, and whose faces correspond to chains of faces of~$K$. The barycentric subdivision $K'$ of $K$ is homeomorphic to~$K$. The $n$-simplex~$\Delta_n$ is the set of all subsets of $\{1,2,\dots,n+1\}$, or geometrically the convex hull of the standard basis in~$\R^{n+1}$. If the vertices of a simplex are labeled from the set $V$ we will denote the simplex by~$\Delta^V$.

\medskip

\noindent
\textbf{Joins and deleted joins.} For simplicial complexes $K$ and $L$ on disjoint vertex sets we denote their join by $K * L$, that is, the simplicial complex that as a set of sets is defined by containing all faces of the form $\sigma \cup \tau$ with $\sigma \in K$ and $\tau \in L$. If $K$ and $L$ are not defined on disjoint vertex set we first make them disjoint before taking their join. (In particular, $K*K$ has twice as many vertices as~$K$.) The $q$-fold join of $K$ is denoted by~$K^{*q}$, while $K^{*q}_\Delta$ denotes the $q$-fold deleted join, that is, the subcomplex of $K^{*q}$ that only contains faces $\sigma_1 \cup \dots \cup \sigma_q$ that were pairwise disjoint even before we forced the $q$ copies of $K$ to have disjoint vertex sets. Topologically the join $K * L$ is the quotient space of $K \times L \times [0,1]$ with the identifications $(k,\ell,0) \sim (k',\ell,0)$ and $(k, \ell, 1) \sim (k, \ell', 1)$ for $k,k'  \in K$ and $\ell, \ell' \in L$. Points in the join $K_1 * K_2 * \dots * K_q$ we will denote by $\lambda_1x_1 + \lambda_2x_2 + \dots + \lambda_qx_q$ with $x_i \in K_i$, $\lambda_i \ge 0$, and $\sum \lambda_i =1$. If $\lambda_i = 0$ then the point $\lambda_1x_1 + \lambda_2x_2 + \dots + \lambda_qx_q$ is independent of the choice of~$x_i$.

\medskip

\noindent
\textbf{Independence complexes and neighborhoods.} A set $S$ in a graph $G$ is independent (or stable) if no edge has both endpoints in~$S$. The set of all independent sets in $G$ is a simplicial complex, called the independence complex~$\Ind(G)$. For any vertex $v$ we denote the size of its neighborhood, that is, the set of all vertices that share an edge with~$v$, by~$N(v)$. The number of vertices that are two edges removed from~$v$ but not in the neighborhood of~$v$ is denoted by~$N^2(v)$. The maximum over all $N(v)$ is the maximal degree, denoted by~$\Delta(G)$.

\medskip

\noindent
\textbf{Equivariant maps.} Given two topological spaces $X$ and $Y$ with actions by the group~$G$, we say that a continuous map $f\colon X \longrightarrow Y$ is $G$-equivariant (or a $G$-map) if $f(g\cdot x) = g\cdot f(x)$ for all $g\in G$ and all~${x \in X}$. We denote the symmertric group on $q$ letters by~$\mathfrak S_q$. By $W_q$ we denote the vector space $\{(y_1, \dots, y_q) \in \R^q \: : \: \sum y_i = 0\}$ with the action by $\mathfrak S_q$ that permutes coordiantes. The $n$-fold direct sum of $W_q$ will be denoted by~$W_q^{\oplus n}$. The group $\mathfrak S_q$ acts diagonally on~$W_q^{\oplus n}$. The unit sphere in $W_q^{\oplus n}$ (with the induced $\mathfrak S_q$-action) will be denoted by~$S(W_q^{\oplus n})$.

\medskip

\noindent
\textbf{Convex hulls and Tverberg-type results.} Tverberg~\cite{tverberg1966} showed that any $(q-1)(d+1)+1$ points in $\R^d$ can be partitioned into $q$ sets whose convex hulls have a point in common. We will use several times that this result is sharp for every $q$ and~$d$: Any $(q-1)(d+1)$ points in $\R^d$ that are in strong general position do not admit a partition into $q$ sets whose convex hulls all share a common point. Points generically are in strong general position.

\section{Embeddings of simplicial complexes and fair splittings of the cycle}
\label{sec:cycle}

Before delving into details we will present our approach in a simple example: consider the cycle graph $G$ on six vertices, labeled cyclically $1,2, \dots, 6$. Suppose the six vertices are split into two sets $V_1$ and $V_2$ of size three each. An almost fair splitting by two independent sets consists of two sets $S_1$ and $S_2$ that are disjoint and intersect each $V_i$ in one vertex. (In fact, in this special case the notions of fair splitting and almost fair splitting coincide.) Let $K$ be the complete bipartite graph for the bipartition $V_1 \sqcup V_2$. Think of the six vertices of $G$ distributed along the unit circle in the plane in cyclic order, and draw in the edges of $K$ as straight line segments. Since $K$, as the complete bipartite graph on $3+3$ vertices, is non-planar, there are two vertex-disjoint edges $S_1$ and $S_2$ of $K$ that intersect. Since $S_1$ and $S_2$ intersect their endpoints are alternating along the circle, and thus $S_1$ and $S_2$ are independent sets in~$G$. They form an almost fair splitting by construction of the graph~$K$.

We will generalize this reasoning to cycles of arbitrary length with partitions of their vertex sets into an arbitrary number of sets of odd cardinality. We will first give a new proof of a special case of a recent result of Alishahi and Meunier~\cite{alishahi2017} for almost fair splittings of cycles, before generalizing our approach to sparse graphs. Alishahi and Meunier show more than the result below; see Remark~\ref{rem:AM}.

\begin{theorem}
\label{thm:cycle}
	For any partition $V_1 \sqcup V_2 \sqcup \dots \sqcup V_m$ of the vertex set $\{1,2,\dots,n\}$ of the $n$-cycle into sets $V_i$ of odd cardinality there is a balanced almost fair splitting by two independent sets.
\end{theorem}

To prove this result we need to adapt the reasoning above. Replace the complete bipartite graph on $3+3$ vertices with the simplicial complex whose facets contain the appropriate amount of vertices of each~$V_j$. The unit circle will be replaced by the moment curve $\gamma(t) = (t, t^2, \dots, t^{2d})$ in~$\R^{2d}$. The non-planarity of $K_{3,3}$ that concluded our reasoning for the $6$-cycle was generalized by Sarkaria~\cite{sarkaria1991}. 

\begin{theorem}[Sarkaria~\cite{sarkaria1991}]
\label{thm:sarkaria}
	Let $k_1, \dots, k_m$ be nonnegative integers, $d = k_1+ \dots + k_m -1$, and let $K = \Delta_{2k_1}^{(k_1-1)} * \dots * \Delta_{2k_m}^{(k_m-1)}$. Then for any continuous map $f\colon K \longrightarrow \R^{2d}$ there are two disjoint faces $\sigma_1$ and $\sigma_2$ of $K$ such that $f(\sigma_1) \cap f(\sigma_2) \ne \emptyset$.
\end{theorem}

Simplified proofs of this and other nonembeddability results can be found in~\cite{blagojevic2014}. With this result, we are now in a position to prove Theorem~\ref{thm:cycle}.

\begin{proof}[Proof of Theorem~\ref{thm:cycle}]
Construct the simplicial complex $K$ on vertex set~$\{1,2,\dots,n\}$, where the set $\sigma \subset \{1,2,\dots,n\}$ is a face of $K$ if and only if $|\sigma \cap V_j| \le  \lfloor\frac{|V_j|+1}{2}\rfloor-1$ for all~$j$. Thus the dimension of $K$ is $d = (\sum_{j=1}^m \lfloor\frac{|V_j|+1}{2}\rfloor-1)-1$, and $K$ is isomorphic to $\Delta_{2k_1}^{(k_1-1)} * \dots * \Delta_{2k_m}^{(k_m-1)}$, where $k_j = \lfloor\frac{|V_j|+1}{2}\rfloor-1$. Here we used that $V_j$ has odd cardinality for all~$j$. Map the vertices of the $n$-cycle in cyclic order to pairwise distinct points along the moment curve $\gamma(t) = (t, t^2, \dots, t^{2d})$ in~$\R^{2d}$. By interpolating linearly we obtain a continuous map $f\colon K \longrightarrow \R^{2d}$ that maps a face $\sigma$ of $K$ to the convex hull of its vertices on~$\gamma$. By Theorem~\ref{thm:sarkaria} there are two disjoint faces $\sigma_1$ and $\sigma_2$ of $K$ such that $f(\sigma_1) \cap f(\sigma_2) \ne \emptyset$. We identify both $\sigma_1$ and $\sigma_2$ with their respective sets of vertices. That the intersection $f(\sigma_1) \cap f(\sigma_2)$ is nonempty means that the convex hulls of $\sigma_1$ and of $\sigma_2$ intersect. As points on the moment curve are in general position, both convex hulls must have dimension~$d$.

Now if two point sets, each of size $d+1$, on the moment curve in $\R^{2d}$ have intersecting convex hulls then their vertices alternate along the curve by Gale's evenness criterion; see~\cite{gale1963}. In particular, both $\sigma_1$ and $\sigma_2$ are independent sets in the $n$-cycle. Since both $\sigma_1$ and $\sigma_2$ have dimension $d = \dim K$ they contain the correct amount of vertices from each~$V_j$. At most one vertex in each $V_j$ (in fact, exactly one vertex) is not contained in $S_1$ or $S_2$ since $|V_j|$ is odd.
\end{proof}

\begin{remark}
\label{rem:AM}
	Alishahi and Meunier~\cite{alishahi2017} establish the existence of a balanced almost fair splitting by two independent sets $S_1$ and $S_2$ for any partition $V_1 \sqcup V_2 \sqcup \dots \sqcup V_m$ of the vertex set of the $n$-cycle if $m$ and $n$ have the same parity. This is a simple corollary of the main result of~\cite{alishahi2017}, and we give a different proof of this main result; see Theorem~\ref{thm:kneser}.
\end{remark}

Adding only two edges to a path or cycle might make it impossible to find an almost fair splitting by two independent sets. Consider a path $P$ and add two edges to $P$ to form two vertex-disjoint $3$-cycles. Now consider a partition of the vertex set where all six vertices of these $3$-cycles are in one part, say~$V_1$, of the partition. Then each independent set can contain at most one vertex from each $3$-cycle, so two independent sets will leave at least two vertices of $V_1$ uncovered. But an almost fair splitting can leave out at most one vertex from each~$V_j$.

\section{Proof of the main results and some extensions}
\label{sec:main}

We will first prove Theorem~\ref{thm:main}. We are given a graph $G$ whose vertex set $V$ is partitioned into $V_1 \sqcup V_2 \sqcup \dots \sqcup V_m$, and an integer $q \ge 2$ such that $|V| = (q-1)n+1$ for some integer $n \ge m+1$ and $|V_j| = qk_j-1$ for integers $k_1, \dots, k_m \ge 1$. We know that for any continuous map $F\colon \Ind(G) \longrightarrow \R^{n-1}$ there are $q$ pairwise disjoint faces $\sigma_1, \dots, \sigma_q$ of $\Ind(G)$ with $F(\sigma_1) \cap \dots \cap F(\sigma_q) \ne \emptyset$. We would like to show that there are $q$ pairwise disjoint independent sets $S_1, \dots, S_q$ in~$G$ such that $|S_i \cap V_j| = k_j -1$ for all $i$ and all~$j$.

\begin{proof}[Proof of Theorem~\ref{thm:main}]
	For each set $V_j$ define $\Sigma_j \subset \Delta^V$ as the subcomplex of faces $\sigma$ with $|\sigma \cap V_j| \le k_j-1$. Given $q$ pairwise disjoint faces $\sigma_1, \dots, \sigma_q$ of $\Delta^V$ at least one face $\sigma_i$ is contained in~$\Sigma_j$ by the pigeonhole principle. In the language of~\cite{blagojevic2014}, this means that $\Sigma_j$ is Tverberg unavoidable. By the proof technique of~\cite[Theorem~4.3]{blagojevic2014} for any continuous map $f\colon \Ind(G) \longrightarrow \R^{n-m-1}$ there are $q$ pairwise disjoint faces $\sigma_1, \dots, \sigma_q$ of $\Ind(G) \cap \Sigma_1 \cap \dots \cap \Sigma_m$ such that $f(\sigma_1) \cap \dots \cap f(\sigma_q) \ne \emptyset$: Namely, consider the map $F\colon \Ind(G) \longrightarrow \R^{n-1}, x \mapsto (f(x), \dist(x, \Sigma_1), \dots, \dist(x, \Sigma_m))$, where the distance $\dist$ from a point to a set is defined using some metric on $\Ind(G)$ that makes $x \mapsto \dist(x, \Sigma_j)$ a continuous function. There are $x_1, \dots, x_q$ in pairwise disjoint faces of $\Ind(G)$ with $F(x_1) = F(x_2) = \dots = F(x_q)$. Since $\Sigma_j$ is Tverberg unavoidable there is an $x_i$ that is in~$\Sigma_j$, so $\dist(x_i, \Sigma_j) = 0$. But then since $\dist(x_1, \Sigma_j) = \dist(x_2, \Sigma_j) = \dots = \dist(x_q, \Sigma_j)$, all distances vanish, and thus $x_1, \dots, x_q \in \Sigma_j$. This is true for all~$j$, so $x_1, \dots, x_q \in \Sigma_1 \cap \dots \cap \Sigma_m$, and $f(x_1) = f(x_2) = \dots = f(x_q)$ since the same holds for~$F$.
	
	Now let $f\colon \Ind(G) \longrightarrow \R^{n-m-1}$ be a strong general position map, that is, whenever $f(\sigma_1) \cap \dots \cap f(\sigma_q) \ne \emptyset$ for pairwise disjoint faces $\sigma_1, \dots, \sigma_q$ of $\Ind(G)$, then they involve at least $(q-1)(n-m) + 1$ vertices. Let $S_1, \dots, S_q$ be pairwise disjoint faces of $\Ind(G) \cap \Sigma_1 \cap \dots \cap \Sigma_m$ with $f(S_1) \cap \dots \cap f(S_q) \ne \emptyset$. Then the $S_i$ are pairwise disjoint independent sets with $|S_i \cap V_j| \le k_j -1$ for all~$i$ and all~$j$. In particular, $\sum_{i=1}^q |S_i \cap V_j| \le qk_j-q$, and so $\bigcup_i S_i$ intersects $V_j$ in all but at least $q-1$ points. This implies that $\bigcup_i S_i$ does not contain at least $(q-1)m$ points of~$V$, and thus the $S_i$ involve at most $(q-1)(n-m)+1$ vertices. However, since $f$ is a strong general position map the $S_i$ have to involve at least $(q-1)(n-m)+1$ vertices, which implies $|S_i \cap V_j| = k_j -1$ for all~$i$ and all~$j$.
\end{proof}

\begin{remark}
    In the proof above we obtain more than just a balanced almost fair splitting by $q$ independent sets $S_1, \dots, S_q$, namely for a fixed $j\in \{1,2,\dots,m\}$ each $S_i$ contains precisely the same amount of vertices of~$V_j$. 
\end{remark}

We now have two goals. First, we would like to remove the condition that $|V_j| = qk_j -1$, and second, we will present a configuration space -- test map scheme for the problem of almost fair splittings by independent sets; see \v Zivaljevic~\cite{zivaljevic1996, zivaljevic1998} and Matou\v sek~\cite{matousek2008} for an introduction to the configuration space -- test map scheme. To any graph $G$ we will associate a simplicial complex $K$ with $\mathfrak S_q$-action such that if $K$ does not admit an $\mathfrak S_q$-map into a certain representation sphere of~$\mathfrak S_q$, we can conclude that $G$ admits an almost fair splitting by $q$ independent sets. These two goals are achieved by Theorem~\ref{thm:main-eq}.

In the proof of Theorem~\ref{thm:main} we employed methods developed in~\cite{blagojevic2014}. These were extended by the third author in~\cite{frick2017-2}. The lemma below is a slight generalization of~\cite[Lemma~2.10]{frick2017-2}. The same proof idea works, which we reproduce here for the reader's convenience.

\begin{lemma}
\label{lem:eq-constraint}
	Let $q\ge 2$, $t\ge 1$, and $k \ge \min\{t,2\}$ be integers, and let $\Delta$ be the simplex on $qk-t$ vertices. Denote by $\Sigma \subset \Delta^{*q}_\Delta$ the subcomplex of faces $\sigma_1 * \dots * \sigma_q$ of $\Delta^{*q}_\Delta$ with $|\sigma_i| \le k-1$ for all~$i$ and $|\sigma_i| \le k-2$ for at least $t-1$ of the~$\sigma_i$. Then there is an $\mathfrak S_q$-equivariant map $\Phi\colon \Delta^{*q}_\Delta \longrightarrow W_q$ with $\Phi^{-1}(0) = \Sigma$.
\end{lemma}

\begin{proof}
	The map $\Phi$ will be defined as an affine map on the barycentric subdivision~$(\Delta^{*q}_\Delta)'$ of~$\Delta^{*q}_\Delta$, that is, we need to decide the value of $\Phi$ on any vertex of $(\Delta^{*q}_\Delta)'$, or equivalently, on any face of~$\Delta^{*q}_\Delta$. Any vertex subdividing a face of $\Sigma$ will be mapped to zero. To determine the value of $\Phi$ on a vertex $v$ that subdivides a face $\sigma_1 * \dots * \sigma_q$ that is not contained in~$\Sigma$, we first need to linearly order the vertices of~$\Delta$ in an arbitrary fashion. Define $\Phi(v)$ to be the standard basis vector $e_j \in \R^q$, where $j$ is the index with $\dim \sigma_j < \dim \sigma_i$ for all $i \ne j$. If multiple $\sigma_j$ have the lowest dimension among $\sigma_1, \dots, \sigma_q$ we use the linear order on the vertices of~$\Delta$ as a tie-breaker. More precisely, if $\dim \sigma_{j_1} = \dim \sigma_{j_2} = \dots = \dim \sigma_{j_\ell} < \dim \sigma_i$ for some indices $j_1, \dots, j_{\ell}$ and all $i \in \{1,2,\dots, q\} \setminus \{j_1, \dots, j_\ell\}$, then let $\Phi(v)$ be equal to $e_{j_t}$, where among the vertices of $\sigma_{j_1} \cup \dots \cup \sigma_{j_\ell}$ the face $\sigma_{j_t}$ has the vertex that comes first in the linear order of all vertices of~$\Delta$.
	
	Denote by $D = \{(y_1, \dots, y_q) \in \R^q \: : \: y_1 = y_2 = \dots = y_q\}$ the diagonal in~$\R^q$. We claim that $\Phi^{-1}(D) = \Sigma$. Certainly $\Sigma \subset \Phi^{-1}(D)$, since $\Phi$ is constantly equal to zero on~$\Sigma$ and $0 \in D$. We need to show that if some point $x \in \Delta^{*q}_\Delta$ satisfies $\Phi(x) \in D$ then $x \in \Sigma$. If $x \notin \Sigma$ then $x$ is in the relative interior of some face $\sigma$ of $(\Delta^{*q}_\Delta)'$ that has a vertex $v$ not contained in~$\Sigma$. In particular, $\Phi(v) = e_j$ for some index~$j$, and thus the $j$th coordinate of $\Phi(x)$ is strictly positive, implying that $\Phi(x) \ne 0$.
	
	Suppose now that $\Phi(x) \in D \setminus \{0\}$. We claim that this leads to a contradiction. Since all coordinates of $\Phi(x)$ are strictly positive, the minimal face $\sigma$ of $(\Delta^{*q}_\Delta)'$ containing~$x$ has vertices $v_1, \dots, v_q$ with $\Phi(v_i) = e_i$. The vertices $v_i$ subdivide faces $\tau_i = \sigma_1^{(i)} * \dots * \sigma_q^{(i)}$ of~$\Delta^{*q}_\Delta$. Since the vertices $v_1, \dots, v_q$ form a face of the barycentric subdivision of~$\Delta^{*q}_\Delta$ the corresponding faces $\tau_1, \dots, \tau_q$ are totally ordered by inclusion. Suppose $\tau_j$ is inclusion-minimal among $\tau_1, \dots, \tau_q$ and $\tau_{j'}$ is inclusion-maximal. The face $\sigma_i^{(i)}$ has the lowest dimension among $\sigma_1^{(i)}, \dots, \sigma_q^{(i)}$ because $\Phi(v_i) = e_i$. Moreover, none of these faces $\tau_i$ are contained in~$\Sigma$. This leads to a contradiction. If the inclusion-minimal face $\tau_j = \sigma_1^{(j)} * \dots * \sigma_q^{(j)}$ has $|\sigma_j^{(j)}| \ge k$, then for the inclusion-maximal face $\tau_{j'} = \sigma_1^{(j')} * \dots * \sigma_q^{(j')}$ all $\sigma_i^{(j')}$ satisfy $|\sigma_i^{(j')}| \ge k$. Thus since the $\sigma_i^{(j')}$ are pairwise disjoint $\sigma_1^{(j')} * \dots * \sigma_q^{(j')}$ involves at least $qk$ vertices, which is a contradiction to $\Delta$ only having $qk-t$ vertices. If on the other hand the inclusion-minimal face $\sigma_1^{(j)} * \dots * \sigma_q^{(j)}$ is not contained in $\Sigma$ since $|\sigma_i^{(j)}| \ge k-1$ for at least $q-t+1$ of the~$\sigma_i^{(j)}$, then for the inclusion-maximal face $\sigma_1^{(j')} * \dots * \sigma_q^{(j')}$ we have that $|\sigma_i^{(j')}| \ge k-1$ for all~$i$ and $|\sigma_i^{(j')}| \ge k$ for at least $q-t+1$ of the~$\sigma_i^{(j')}$. This again leads to the contradiction that the~$\sigma_i^{(j')}$ involve more than $qk-t$ vertices.
	
	We have constructed an $\mathfrak S_q$-equivariant map $\Phi\colon \Delta^{*q}_\Delta \longrightarrow \R^q$ that maps precisely the points in $\Sigma$ to the diagonal~$D$. The desired map to $W_q$ can now be constructed by orthogonally projecting along~$D$ onto $D^\perp = W_q$.
\end{proof}

Let us recall the statement of Theorem~\ref{thm:main-eq}. We are given a graph $G$ whose vertex set~$V$ is partitioned into $V_1 \sqcup V_2 \sqcup \dots \sqcup V_m$, as well as integers $n \ge m+1$ and $q \ge 2$ with $|V_j| \ge q-1$ for all~$j$ and $|V| \le (q-1)n+1$ such that there is no $\mathfrak S_q$-equivariant map $\Ind(G)^{*q}_\Delta \longrightarrow S(W_q^{\oplus n})$. Our goal is to show that $G$ admits an almost fair splitting by $q$ independent sets.

\begin{proof}[Proof of Theorem~\ref{thm:main-eq}]
    For each set $V_j$ find integers $t_j \in \{1,2,\dots, q\}$ and $k_j \ge \min\{2,t_j\}$ with $|V_j| = qk_j-t_j$. Following Lemma~\ref{lem:eq-constraint}, denote by $\Sigma_j \subset (\Delta^{V_j})^{*q}_\Delta$ the subcomplex of faces $\sigma_1 * \dots * \sigma_q$ of $(\Delta^{V_j})^{*q}_\Delta$ with $|\sigma_i| \le k_j-1$ for all~$i$ and $|\sigma_i| \le k_j-2$ for at least $t_j-1$ of the~$\sigma_i$. In particular, $|\bigcup_i \sigma_i| \le q(k_j-1)-t_j+1$.

    Let $f\colon \Ind(G) \longrightarrow \R^{n-m-1}$ be an affine map that maps the vertices of $\Ind(G)$ to points in strong general position. In particular, $f$ is generic in the sense that if pairwise disjoint faces $\sigma_1, \dots, \sigma_q$ of $\Ind(G)$ satisfy $f(\sigma_1) \cap \dots \cap f(\sigma_q) \ne \emptyset$ then the faces $\sigma_i$ involve at least $(q-1)(n-m)+1$ vertices. Define an $\mathfrak S_q$-equivariant map $F\colon \Ind(G)^{*q}_\Delta \longrightarrow W_q^{\oplus (n-m)}$ in the following way: Compose the $\mathfrak S_q$-equivariant map (here $\mathfrak S_q$ acts on the codomain by permuting coordinates) $$\Ind(G)^{*q}_\Delta \longrightarrow (\R^{n-m})^q, \lambda_1x_1 + \dots + \lambda_qx_q \mapsto (\lambda_1, \lambda_1f(x_1), \dots, \lambda_q, \lambda_qf(x_q))$$ with the projection along the diagonal $D = \{(y_1, \dots, y_q) \in (\R^{n-m})^q \: : \: y_1 = \dots = y_q\}$ onto the orthogonal complement~$D^\perp$, which is equivariantly isomorphic to~$W_q^{\oplus(n-m)}$. Notice that $F(\lambda_1x_1 + \dots + \lambda_qx_q) = 0$ if and only if $(\lambda_1, \lambda_1f(x_1), \dots, \lambda_q, \lambda_qf(x_q)) \in D$ if and only if $\lambda_1 = \lambda_2 = \dots = \lambda_q$ and $f(x_1) = f(x_2) = \dots = f(x_q)$. 

    The independence complex $\Ind(G)$ is a subcomplex of $\Delta^V$ and thus $\Ind(G)^{*q}_\Delta \subset (\Delta^V)^{*q}_\Delta$. For each $V_j \subset V$ we have an $\mathfrak S_q$-map $\Phi_j\colon (\Delta^{V_j})^{*q}_\Delta \longrightarrow W_q$ with $\Phi_j^{-1}(0) = \Sigma_j$. Define the map $\Phi\colon (\Delta^{V_1})^{*q}_\Delta * \dots * (\Delta^{V_m})^{*q}_\Delta \longrightarrow W_q^{\oplus m}$ by $\Phi(\lambda_1x_1 + \dots + \lambda_mx_m) = (\lambda_1\Phi_1(x_1), \dots, \lambda_m\Phi_m(x_m))$. The concrete choice of $x_i$ does not matter for $\lambda_i = 0$, so $\Phi$ is well-defined and continuous. The complex $(\Delta^{V_1})^{*q}_\Delta * \dots * (\Delta^{V_m})^{*q}_\Delta$ is isomorphic to $(\Delta^{V_1} * \dots * \Delta^{V_m})^{*q}_\Delta$, which simply is $(\Delta^V)^{*q}_\Delta$. The map $\Phi$ is zero precisely on the subcomplex $\Sigma_1 * \dots * \Sigma_m$.

    The $\mathfrak S_q$-map $F\oplus \Phi\colon \Ind(G)^{*q}_\Delta \longrightarrow W_q^{\oplus n}, x \mapsto (F(x), \Phi(x))$ must have a zero, since otherwise we could equivariantly retract to the unit sphere to obtain an $\mathfrak S_q$-map $\Ind(G)^{*q}_\Delta \longrightarrow S(W_q^{\oplus n})$. Thus there is a point $\lambda_1x_1 + \dots + \lambda_qx_q \in \Ind(G)^{*q}_\Delta \cap (\Sigma_1 * \dots * \Sigma_m)$ with $F(\lambda_1x_1 + \dots + \lambda_qx_q) = 0$. The latter means that $f(x_1) = \dots = f(x_q)$. Let $S_i$ denote the minimal face of $\Ind(G)$ that $x_i$ is contained in. Suppose that for some $j$ the intersection $\bigcup_i S_i \cap V_j$ had size at most $q(k_j-1)-t_j$, that is, the faces $S_i$ miss at least $q$ vertices of~$V_j$. Then $\bigcup_i S_i$ does not contain at least $(q-1)m+1$ vertices of~$V$. Since the size of $V$ is at most $(q-1)n+1$, the $S_i$ involve at most $(q-1)(n-m)$ vertices. This is a contradiction to $f$ being a general position map. Thus, $\bigcup_i S_i \cap V_j$ has size $q(k_j-1)-t_j+1$ for all~$j$, and so $|S_i \cap V_j| = k_j-2$ for $t_j-1$ of the~$S_i$, and $|S_i \cap V_j| = k_j-1$ for the other~$S_i$. If $t_j=1$ for some $j$ then all $S_i \cap V_j$ have size $k_j-1 = \lfloor\frac{(qk_j-1)+1}{q}\rfloor-1$, wheras if $t_j > 1$ then all $S_i \cap V_j$ have size at least $k_j-2 = \lfloor\frac{(qk_j-t_j)+1}{q}\rfloor-1$, and $\bigcup_i S_i$ covers all but $q-1$ vertices of~$V_j$.
\end{proof}

In the proof above we argued using a strong general position map $f\colon \Ind(G) \longrightarrow \R^{n-m-1}$. By using specific such maps $f$ we can slightly enlarge the graph $G$ and still find an almost fair splitting by $q$ independent sets. That is, we can augment the equivariant topological approach above by further constraining the independent sets using the intersection combinatorics of convex sets in Euclidean space, similar to the third author's work on chromatic numbers of hypergraphs~\cite{frick2017, frick2018}. The following remark is a simple consequence of the proof above.

\begin{remark}
    If whenever the intersection $f(\sigma_1) \cap \dots \cap f(\sigma_q)$ consists of exactly one point for pairwise disjoint faces $\sigma_1, \dots, \sigma_q$ of~$\Ind(G)$, we can guarantee that no face $\sigma_i$ contains both vertex $v$ and $w$ of~$G$, we may add the edge $(v,w)$ to~$G$. The graph $G'$ that is obtained from $G$ by adding all such edges still admits an almost fair splitting by $q$ independent sets provided that there is no $\mathfrak S_q$-equivariant map $\Ind(G)^{*q}_\Delta \longrightarrow S(W_q^{\oplus n})$. This is simply because the almost fair splitting by independent sets consists of $q$ faces $\sigma_1, \dots, \sigma_q$ of $\Ind(G)$ with $f(\sigma_1) \cap \dots \cap f(\sigma_q) \ne \emptyset$, but no face $\sigma_i$ can contain an edge of~$G'$, so the sets $\sigma_i$ are independent in~$G'$ too. 
\end{remark}

For example, by placing points sufficiently far apart along the moment curve we can derive the following result (Theorem~\ref{lem:weakly-stable-points} gives the precise intersection combinatorics for such point sets):

\begin{theorem}
\label{thm:main+path}
	Let $G$ be a graph on vertex set~$V$, and let $V_1 \sqcup V_2 \sqcup \dots \sqcup V_m$ be a partition of~$V$. Let $n \ge m+1$ and $q \ge 2$ be integers such that $|V_j| \ge q-1$ for all~$j$, $|V| \le (q-1)n+1$, and such that there is no $\mathfrak S_q$-equivariant map $\Ind(G)^{*q}_\Delta \longrightarrow S(W_q^{\oplus n})$. Let $H$ be a graph on vertex set $V$ obtained by adding a simple path to~$G$. Then $H$ admits an almost fair splitting by $q$ independent sets.
\end{theorem}

\begin{remark}
\label{rem:comparison}
	Theorem~\ref{thm:main-eq} is a proper strenghtening of Theorem~\ref{thm:main}. If there is a continuous map $F\colon \Ind(G) \longrightarrow \R^{n-1}$ such that for  every collection of $q$ pairwise disjoint faces $S_1, \dots, S_q$ of $\Ind(G)$ the intersection $F(S_1) \cap \dots \cap F(S_q)$ is empty, then the $\mathfrak S_q$-map
	$$\Phi\colon \Ind(G)^{*q}_\Delta \longrightarrow (\R^n)^q, \lambda_1x_1 + \dots + \lambda_qx_q \mapsto (\lambda_1, \lambda_1F(x_1), \dots, \lambda_q, \lambda_qF(x_q))$$
	misses the diagonal $D = \{(y_1, \dots, y_q) \in (\R^n)^q \: : \: y_1 = \dots = y_q\}$ and thus can be orthogonally projected to $D^\perp = W_q^{\oplus n}$, where no point is mapped to the origin, so we can further equivariantly retract to~$S(W_q^{\oplus n})$. 
	
	Theorem~\ref{thm:main} furthermore guarantees that the independent sets $S_1, \dots, S_q$ are balanced. In the case that $|V_j| = qk_j -1$ the proof of Theorem~\ref{thm:main-eq} establishes that all $S_i$ have the same size.
\end{remark}

\section{Consequences of the main results}
\label{sec:consequences}

In this section we will derive some consequences of Theorem~\ref{thm:main+path}, which combines the nonexistence of equivariant maps with understanding intersection patterns of convex hulls to find almost fair splittings by independent sets in graphs that are denser than paths or cycles. One advantage of our approach is that the relevant configuration spaces $\Ind(G)^{*q}_\Delta$ are the same as for Tverberg-type problems. This is a well-studied collection of problems that aim to characterize simplicial complexes that have $q$-fold points of coincidence for any continuous map to Euclidean space of a fixed dimension. For any Tverberg-type theorem established via the topological configuration space machinery we thus get a corresponding result on almost fair splittings by independent sets. We give two examples of this phenomenom and cite two results on the nonexistence of equivariant maps that were originally used to establish Tverberg-type results:

\begin{theorem}[Engstr\"om~\cite{engstrom2011}]
\label{thm:constraint-graph}
	Let $q \ge 2$ be a prime power. Let $G$ be a graph on at least ${(q-1)n+1}$ vertices with $2N(v) + N^2(v) < q$ for every vertex~$v$. Then there is no $\mathfrak S_q$-equivariant map ${\Ind(G)^{*q}_\Delta \longrightarrow S(W_q^{\oplus n})}$.
\end{theorem}

\begin{theorem}[Blagojevi\'c, Matschke, and Ziegler~\cite{blagojevic2009}]
\label{thm:optimal}
	Let $q \ge 2$ be a prime, and let $G$ be a graph on ${(q-1)n+1}$ vertices that is the disjoint union of $n$ cliques of size~${q-1}$ and an isolated vertex. Then there is no $\mathfrak S_q$-equivariant map ${\Ind(G)^{*q}_\Delta \longrightarrow S(W_q^{\oplus n})}$.
\end{theorem}

With these two results on the nonexistence of $\mathfrak S_q$-maps we can now deduce the corollaries advertised in the introduction.

\begin{theorem}
\label{thm:corollaries}
    Let $G$ be a graph whose vertex set~$V$ is partitioned into $V_1 \sqcup V_2 \sqcup \dots \sqcup V_m$.
    \begin{compactenum}[(a)]
        \item\label{it:sparse+path} Suppose there is a prime power $q \ge 2$ with $|V_j| \ge q-1$ for all~$j$ and $|V| \ge (q-1)(m+2)+1$. Suppose further that after deleting the edges of a simple path from $G$ every vertex $v \in V$ satisfies $2N(v) + N^2(v) < q$. Then $G$ admits an almost fair splitting by $q$ independent sets.
        \item Let $q \ge 2$ be a prime with $|V_j| \ge q-1$ for all~$j$. Let $G$ be the edge-disjoint union of a path and pairwise vertex-disjoint cliques of size~${q-1}$ on $(q-1)n+1$ vertices for some integer $n \ge m+1$. Then $G$ admits an almost fair splitting by $q$ independent sets.
    \end{compactenum}
\end{theorem}

\begin{proof}
	\begin{compactenum}[(a)]
		\item If the number of vertices $|V|$ is of the form $(q-1)n+1$ for some integer~$n$, then this is an immediate consequence of combining Theorem~\ref{thm:main+path} with Theorem~\ref{thm:constraint-graph}. Otherwise add between $q$ and $2q-2$ vertices to $V$ as a new part $V_{m+1}$ such that $|V|$ is of this form. Now Theorem~\ref{thm:main+path} guarantees the existence of an almost fair splitting by $q$ independent sets $S_1, \dots, S_q$. After deleting any vertices of $V_{m+1}$ from these sets~$S_i$, they are still independent in~$G$ and an almost fair splitting.
		\item This follows immediately by combining Theorem~\ref{thm:main+path} and Theorem~\ref{thm:optimal}.\qedhere
	\end{compactenum}
\end{proof}

A precursor to fair representation results is a result of Haxell~\cite{haxell1995, haxell2001} about the existence of an independent set that intersects each part~$V_j$, extending earlier results of Alon~\cite{alon1988, alon1994}.

\begin{theorem}[Haxell~\cite{haxell2001}]
	Let $V_1 \sqcup V_2 \sqcup \dots \sqcup V_m$ be a partition of the vertex set of a graph~$G$ such that $|V_j| \ge 2\Delta(G)$ for all~$j$. Then there exists an independent set $S$ that intersects all~$V_j$.
\end{theorem}

We can use Theorem~\ref{thm:corollaries} to prove the existence of multiple pairwise disjoint independent sets that each intersect each part~$V_j$ in a more restrictive setting.

\begin{theorem}
\label{thm:several-transversals}
	Let $V_1 \sqcup V_2 \sqcup \dots \sqcup V_m$ be a partition of the vertex set of a graph~$G$. Let $q$ be a prime power with $q > 2N(v) + N^2(v)$ for all vertices $v$ of~$G$. If $|V_j| \ge 2q-1$ for all~$j$, then there are $q$ pairwise disjoint independent sets $S_1, \dots, S_q$ such that each $S_i$ intersects each~$V_j$.
\end{theorem}

\begin{proof}
	The inequality $|V_j| \ge 2q-1$ implies that $\lfloor\frac{|V_{m+1}|+1}{q}\rfloor-1 \ge 1$. In particular, in an almost fair splitting by independent sets $S_1, \dots, S_q$, each $S_i$ intersects each~$V_j$. Now the result is an immediate consequence of Theorem~\ref{thm:corollaries}(\ref{it:sparse+path}).
\end{proof}

\section{Stable almost fair splittings of a path}
\label{sec:path}

A set of vertices $S$ in a graph is \emph{$q$-stable} if any two distinct vertices in $S$ are at distance at least~$q$ in the graph. A conjecture of Alishahi and Meunier about special almost fair splittings of a path is stated in~\cite{alishahi2017} as follows:

\begin{conjecture}[Alishahi and Meunier]
\label{conj:AM}
    Given a positive integer $q$ and a path $P$ whose vertex set is partitioned into $m$ subsets $V_1, \dots, V_m$ of sizes at least $q - 1$, there always exist pairwise disjoint $q$-stable sets $S_1,\dots , S_q$ covering all vertices but $q - 1$ in each $V_j$, with sizes differing by at most one, and satisfying $$|S_i \cap V_j| \ge \Big\lfloor \frac{|V_j| + 1}{q} \Big\rfloor - 1$$ for all $i \in \{1,2,\dots,q\}$ and all $j \in \{1,2,\dots,m\}$.
\end{conjecture}

Let $G$ be the graph on vertex set $\{1,2,\dots,n\}$ where two distinct vertices $v$ and $w$ are joined by an edge if and only if $|v-w| < q$. Then the independent sets of $G$ are exactly the $q$-stable sets of~$P$. In our language Conjecture~\ref{conj:AM} can then be stated as: For any partition $V_1 \sqcup V_2 \sqcup \dots \sqcup V_m$ of $\{1,2,\dots,n\}$ there is a balanced almost fair splitting by $q$ independent sets in~$G$.

As mentioned in Remark~\ref{rem:AM}, Alishahi and Meunier establish Conjecture~\ref{conj:AM} for~${q=2}$. Moreover, they show that if the conjecture holds for $q'$ and $q''$ then it also holds for their product $q = q'q''$. Thus Conjecture~\ref{conj:AM} holds for all powers of two. We will give an alternative proof for $q = 2^t$ in the next section. First we will show approximations to Conjecture~\ref{conj:AM} for other values of~$q$.

While we are unable to decide Conjecture~\ref{conj:AM} for other values of~$q$, we can show a weaker version of $q$-stability that is better adapted to our geometric method. We call a partition $S_1 \sqcup S_2 \sqcup \dots \sqcup S_q$ of $\{1,2,\dots,(q-1)n+1\}$ \emph{weakly $q$-stable} if $S_i \cap \{(q-1)(k-1)+1, (q-1)(k-1)+2, \dots, (q-1)k+1\}$ consists of a single point for all~$i$ and all~${k \in \{1,2,\dots, n\}}$. That is, there are specified blocks of length~$q$ that contain exactly one point from each~$S_i$. Consecutive blocks overlap in one point. If $S_1, \dots, S_q \subset \{1,2,\dots,N\}$ are pairwise disjoint and involve $(q-1)n+1$ points then we generalize the definition of weakly $q$-stable in the following way: Let $\varphi\colon \bigcup_i S_i \longrightarrow \{1,2,\dots,(q-1)n+1\}$ denote the order-preserving bijection; then $S_1, \dots, S_q$ are \emph{weakly $q$-stable} if $\varphi(S_1) \sqcup \varphi(S_2) \sqcup \dots \sqcup \varphi(S_q)$ is a weakly $q$-stable partition of~$\{1,2,\dots,(q-1)n+1\}$. The main geometric ingredient now is:

\begin{theorem}[Bukh, Loh, and Nivasch~\cite{bukh2017}]
\label{lem:weakly-stable-points}
    There are arbitrarily long sequences $x_1, x_2, \dots, x_N$ of points in $\R^{n-1}$ in strong general position such that $q$ pairwise disjoint sets $S_1, \dots, S_q \subset \{1,2, \dots, N\}$ involving exactly $(q-1)n+1$ points satisfy $\bigcap_i \conv \{x_j \: : \: j \in S_i\} \ne \emptyset$ if and only if $S_1, \dots, S_q$ are weakly $q$-stable.
\end{theorem}

We can use this to show the following weaker version of Conjecture~\ref{conj:AM}.

\begin{theorem}
\label{thm:weakly-stable}
    Let $q \ge 2$ be a prime power, and let $P$ be a path whose vertex set is partitioned into $V_1 \sqcup V_2 \sqcup \dots \sqcup V_m$ with $|V_j| \ge q-1$ for all~$j$. Suppose the length of $P$ is at least $(q-1)(m+2)+1$. Then there are $q$ pairwise disjoint $(\lfloor \frac{q}{6} \rfloor +1)$-stable independent sets $S_1, \dots, S_q$ that are an almost fair splitting of $V_1 \sqcup V_2 \sqcup \dots \sqcup V_m$. Moreover, $S_1, \dots, S_q$ can be chosen to be weakly $q$-stable.
\end{theorem}

\begin{proof}
    Label the vertices of $P$ as $1,2,\dots,N$ such that $P$ traverses them in order. Let $G$ be the graph on vertex set $\{1,2,\dots,N\}$ that connects two distinct vertices $v$ and $w$ by an edge if and only if $|v-w| \le \lfloor \frac{q}{6} \rfloor$. The independent sets of $G$ are exactly the $(\lfloor \frac{q}{6} \rfloor +1)$-stable independent sets of~$P$. A vertex $v$ has degree at most $N(v) \le 2\lfloor \frac{q}{6} \rfloor$ and similarly $N^2(v) \le 2\lfloor \frac{q}{6} \rfloor$. Since $q$ is a prime power and thus not divisible by~six, we have that $2N(v) +N^2(v) < q$ for all vertices~$v$. Now Theorem~\ref{thm:corollaries}(\ref{it:sparse+path}) guarantees the existence of pairwise disjoint $(\lfloor \frac{q}{6} \rfloor +1)$-stable independent sets $S_1, \dots, S_q$ that are an almost fair splitting. 
    
    To see that the $S_i$ can moreover be chosen to be weakly $q$-stable we have to tweak the proof of Theorem~\ref{thm:main-eq}. As in the proof of Theorem~\ref{thm:corollaries}(\ref{it:sparse+path}) we can assume that $N = (q-1)n+1$ for some integer $n \ge m+2$. Let $x_1, \dots, x_N$ be a sequence of points in~$\R^{n-m-1}$ as in Theorem~\ref{lem:weakly-stable-points}. Define the strong general position map $f\colon \Ind(G) \longrightarrow \R^{n-m-1}$ by sending vertex $v \in \{1,2,\dots, N\}$ of $G$ to point~$x_v$. Then extend this map linearly onto the faces of~$\Ind(G)$. The proof of Theorem~\ref{thm:main-eq} now shows the existence of $q$ faces $S_1, \dots, S_q$ of $\Ind(G)$ that are an almost fair splitting and such that $f(S_1) \cap f(S_2) \cap \dots \cap f(S_q) \ne \emptyset$. But this intersection being nonempty precisely means that $S_1, \dots, S_q$ are weakly $q$-stable.
\end{proof}

For $q \le 11$ being $(\lfloor \frac{q}{6} \rfloor +1)$-stable is no improvement over being independent in the path. For prime powers $4 \le q \le 11$ we can use the following result of Hell to establish $3$-stability.

\begin{theorem}[Hell~\cite{hell2008}]
\label{thm:constraint-path}
    Let $q \ge 4$ be a prime power, $n \ge 1$ an integer, and let $G$ be a path of length $(q-1)n+1$. Then there is no $\mathfrak S_q$-equivariant map $I(G)^{*q}_\Delta \longrightarrow S(W_q^{\oplus n})$.
\end{theorem}

It is now simple to derive the following consequence.

\begin{theorem}
    Let $q \ge 4$ be a prime power, and let $P$ be a path whose vertex set is partitioned into $V_1 \sqcup V_2 \sqcup \dots \sqcup V_m$ with $|V_j| \ge q-1$ for all~$j$. Suppose the length of $P$ is at least $(q-1)(m+2)+1$. Then there are $q$ pairwise disjoint $3$-stable independent sets $S_1, \dots, S_q$ that are an almost fair splitting of $V_1 \sqcup V_2 \sqcup \dots \sqcup V_m$. Moreover, $S_1, \dots, S_q$ can be chosen to be weakly $q$-stable.
\end{theorem}

\begin{proof}
	Let $G$ be the graph on vertex set $\{1,2,\dots,(q-1)n+1\}$ for some integer $n\ge m+2$ with edges $(i,i+2)$ for $i \in \{1,2,\dots, {(q-1)n-1}\}$. Then since $G$ is a subgraph of a path there is no $\mathfrak S_q$-map $\Ind(G)^{*q}_\Delta \longrightarrow S(W_q^{\oplus n})$ by Theorem~\ref{thm:constraint-path}. We now proceed as in the proof of Theorem~\ref{thm:weakly-stable}. We only need to observe that if $S_1, \dots, S_q$ are weakly $q$-stable for any $q \ge 2$ then no $S_i$ contains two consecutive vertices. 
\end{proof}

For parameters $q$ that are not prime powers, we can use an induction on the number of distinct prime divisors. This is very similar to~\cite[Prop.~1]{alishahi2017}. We will make use of the equality $$\Big\lfloor \frac{1}{c} \Big\lfloor \frac{a}{b} \Big\rfloor \Big\rfloor = \Big\lfloor\frac{a}{bc} \Big\rfloor$$ for all $a,b,c \in \Z$ that is proven there.

\begin{theorem}
	Let $m, q_1, q_2, s_1$, and $s_2$ be positive integers. Let $q = q_1q_2$ and $s = s_1s_2$. Suppose for both $i = 1$ and $i = 2$ and any path and any partition $V_1 \sqcup \dots \sqcup V_m$ of its vertex set with $|V_j| \ge q_i-1$ for all~$j$ there is an almost fair representation by pairwise disjoint $s_i$-stable sets $A^{(i)}_1, \dots, A^{(i)}_{q_i}$. Then for any path and any partition $V_1 \sqcup \dots \sqcup V_m$ of its vertex set with $|V_j| \ge q-1$ for all~$j$ there is an almost fair representation by pairwise disjoint $s$-stable sets $S_1, \dots, S_q$. If moreover, the sets $A^{(1)}_1, \dots, A^{(1)}_{q_1}$ can be chosen to be weakly $w_1$-stable for some $w_1 \ge 2$, then there is an almost fair representation by pairwise disjoint $[(s_2-1)(w_1-1)+1]$-stable sets $S_1, \dots, S_q$.
\end{theorem}

\begin{proof}
	Let $P$ be a path whose vertex set is partitioned into $V_1 \sqcup \dots \sqcup V_m$ with $|V_j| \ge q-1$ for all~$j$. Find an almost fair representation by $q_1$ pairwise disjoint $s_1$-stable sets $S'_1, \dots, S'_{q_1}$. For each $t \in \{1, 2, \dots, q_1\}$ let $P'_t$ be the path that connects the vertices of $S'_t$ in the same order that they are traversed by~$P$. The vertex set $S'_t$ of $P'_t$ is partitioned into $(S'_t \cap V_1) \sqcup \dots \sqcup (S'_t \cap V_m)$. By definition of almost fair representation 
	$$|S'_t \cap V_j| \ge \Big\lfloor\frac{|V_j| +1}{q_1}\Big\rfloor -1 \ge  \Big\lfloor\frac{q_1q_2}{q_1}\Big\rfloor -1 = q_2 -1$$
	for each~$j$. Thus there is an almost fair representation for the partition $(S'_t \cap V_1) \sqcup \dots \sqcup (S'_t \cap V_m)$ by $s_2$-stable sets $S^{(t)}_1, \dots, S^{(t)}_{q_2}$ in~$P'_t$.
	
	We claim that the collection of sets $S^{(t)}_1, \dots, S^{(t)}_{q_2}$ where $t$ ranges over $\{1,2,\dots,q_1\}$ considered as sets of vertices of~$P$ are an almost fair representation of $V_1 \sqcup \dots \sqcup V_m$ by $s$-stable sets. Certainly, the sets $S^{(t)}_i$ are $s$-stable as $s_2$-stable sets in the $s_1$-stable set~$S'_t$. If the sets $S'_t$ are weakly $w_1$-stable, then between any two consecutive elements of $S^{(t)}_i$ there are $s_2-1$ blocks of size~$w_2$, where consecutive blocks overlap in one vertex. Thus in this case the sets $S^{(t)}_i$ are $[(s_2-1)(w_1-1)+1]$-stable. The collection $S^{(t)}_1, \dots, S^{(t)}_{q_2}$ is an almost fair representation of $(S'_t \cap V_1) \sqcup \dots \sqcup (S'_t \cap V_m)$, so
	$$|S^{(t)}_i \cap V_j| = |S^{(t)}_i \cap (S'_t \cap V_j)| \ge \Big\lfloor\frac{|S'_t \cap V_j| +1}{q_2}\Big\rfloor -1 \ge \Big\lfloor\frac{\lfloor\frac{|V_j| +1}{q_1}\rfloor}{q_2}\Big\rfloor -1 = \Big\lfloor\frac{|V_j| +1}{q_1q_2}\Big\rfloor -1 = \Big\lfloor\frac{|V_j| +1}{q}\Big\rfloor -1$$
	for all~$i$, all~$j$, and all~$t$.
	
	Lastly, for each fixed~$j$ the sets $S'_1, \dots, S'_{q_1}$ cover all but at most $q_1-1$ vertices of~$V_j$. If we now further fix $t \in \{1,2,\dots,q_1\}$, then the sets $S^{(t)}_1, \dots, S^{(t)}_{q_2}$ cover all but at most $q_2-1$ vertices of~${V_j \cap S'_t}$. Thus the collection of sets $S^{(t)}_i$ with $i \in \{1,2,\dots, q_2\}$ and $t \in \{1,2,\dots,q_1\}$ covers all but at most $q_1-1 + q_1(q_2-1) = q-1$ vertices of~$V_j$.
\end{proof}

\section{Relation to Kneser hypergraphs}
\label{sec:kneser}

Aharoni et al.~\cite{aharoni2017} show that almost fair representations by $q$-stable sets for cycles are related to chromatic numbers of certain Kneser hypergraphs. Here we extend their reasoning to show that almost fair \emph{splittings} by $q$-stable sets for paths are also related to Kneser hypergraphs. The relation is not as straightforward as one might hope. In particular, we are unable to extend our arguments to cycles or even more general sparse graphs.

A $q$-uniform hypergraph $H$ is a set of $q$-element subsets of some ground set~$X$. The sets in $H$ are called hyperdges and $X$ is the vertex set of~$H$. The (weak) chromatic number $\chi(H)$ of $H$ is the least number of colors needed to color the vertices $X$ such that every hyperedge of $H$ has elements of at least two distinct colors. A Kneser hypergraph has as vertex set some system of sets $\mathcal F$ and a hyperedge for any $q$ pairwise disjoint sets $F_1, \dots, F_q \in \mathcal F$. The Kneser hypergraphs $\KG^q(n,k)$ whose vertices are the $k$-element subsets of $\{1,2,\dots, n\}$ with hyperedges $\{F_1, \dots, F_q\}$ for pairwise disjoint $k$-element sets have received particular attention. Alon, Frankl, and Lov\'asz~\cite{alon1986} showed that $\chi(\KG^q(n,k)) = \lceil\frac{n-q(k-1)}{q-1}\rceil$ for $n \ge qk$. Ziegler~\cite{ziegler2002} and Alon, Drewnowski, and \L uczak~\cite{alon2009} conjectured that the subhypergraph $\KG^q(n,k)_{q-\mathrm{stab}}$ whose vertex set consists of only those $k$-element subsets $F$ whose elements form a $q$-stable set in the cycle on~$\{1,2,\dots,n\}$ still has the same chromatic number. If $F$ is only required to be $q$-stable in the path then the corresponding Kneser hypergraph is denoted by~$\KG^q(n,k)_{\widetilde{q-\mathrm{stab}}}$. This was introduced by Meunier~\cite{meunier2011} together with the conjecture that the chromatic number of $\KG^q(n,k)_{\widetilde{q-\mathrm{stab}}}$ is $\lceil\frac{n-q(k-1)}{q-1}\rceil$ for $n \ge qk$ as well.

That in fact $\KG^q(n,k)_{q-\mathrm{stab}} = \lceil\frac{n-q(k-1)}{q-1}\rceil$ for $n \ge qk$ is known for $q$ a power of two. For $q=2$ this is a classical result of Schrijver~\cite{schrijver1978}, and Alon, Drewnowski, and Luczak~\cite{alon2009} show that if this is true for $q'$ and $q''$ then it also holds for their product~${q = q'q''}$.

We first formulate a weakening of Conjecture~\ref{conj:AM} that does not require that the almost fair splitting is balanced. We then relate this conjecture to Meunier's conjecture on the chromatic number of~$\KG^q(n,k)_{\widetilde{q-\mathrm{stab}}}$.

\begin{conjecture}
\label{conj:AM-unbalanced}
    Given a positive integer $q$ and a path $P$ whose vertex set is partitioned into $m$ subsets $V_1, \dots, V_m$ of sizes at least $q - 1$, there always exist pairwise disjoint $q$-stable sets $S_1,\dots , S_q$ covering all vertices but $q - 1$ in each $V_j$, and satisfying $$|S_i \cap V_j| \ge \Big\lfloor \frac{|V_j| + 1}{q} \Big\rfloor - 1$$ for all $i \in \{1,2,\dots,q\}$ and all $j \in \{1,2,\dots,m\}$.
\end{conjecture}

We can now adapt the reasoning of Aharoni et al.~\cite{aharoni2017}. As a special case, we obtain a new proof of Conjecture~\ref{conj:AM} for $q$ a power of two.

\begin{theorem}
\label{thm:kneser}
    Let $q \ge 2$ be an integer. If $\chi(\KG^q(n,k)_{\widetilde{q-\mathrm{stab}}}) = \Big\lceil \frac{n-q(k-1)}{q-1} \Big\rceil$ holds for all integers $k \ge 1$ and $n \ge qk$ then Conjecture~\ref{conj:AM-unbalanced} holds for~$q$. For $q=2$ we can moreover guarantee that the independent sets $S_1$ and $S_2$ that form the almost fair splitting are balanced. In particular, Conjecture~\ref{conj:AM} holds for $q$ a power of two.
\end{theorem}

\begin{proof}
    First assume that all sets $V_j$ have size $qk_j-1$ for some integers $k_1, \dots, k_m$. Thus the path has $n = \sum (qk_j-1)$ vertices in total. Let $k = \sum (k_j-1)$. Notice that
    $$\Big\lceil \frac{n-q(k-1)}{q-1} \Big\rceil = \Big\lceil \frac{\sum (qk_j-1)-\Big[\sum (qk_j-q)-q\Big]}{q-1} \Big\rceil = \Big\lceil \frac{m(q-1)+q}{q-1}\Big\rceil = m+2.$$
    
    Let $S \subset \{1,2,\dots,n\}$ be a $q$-stable $k$-element set. Define the color $C(S)$ of~$S$ by $$C(S) = \min (\{j \in \{1,2,\dots,m\} \: : \: |S \cap V_j| \ge k_j\} \cup \{m+1\}).$$ Let $S_1, \dots, S_q$ be a monochromatic hyperedge of~$\KG^q(n,k)_{\widetilde{q-\mathrm{stab}}}$. If they all had the same color $C(S_i) = j \le m$ then this would imply that $V_j$ contains at least $qk_j$ elements, which it does not. Thus $C(S_i) = m+1$ for all~$i$. This implies $|S_i \cap V_j| \le k_j-1$ for all~$i$ and all~$j$. Since all $S_i$ have size $k = \sum (k_j-1)$, we have that $|S_i \cap V_j| = k_j-1$ for all~$i$ and all~$j$. This proves Conjecture~\ref{conj:AM-unbalanced} for this specific case.
    
    It remains to show that it is sufficient to prove Conjecture~\ref{conj:AM-unbalanced} in this specific case. Suppose that $|V_j| = qk_j-t_j$ for integers $k_j$ and $t_j \in \{1,2,\dots, q\}$. Now for each $j$ modify the path by adding a block $B_j$ of $t_j-1$ consecutive vertices to the end of the path. Define an enlarged partition of this elongated path by $V'_j = V_j \cup B_j$. Then each $V'_j$ has size~${qk_j-1}$. For this modified path we can find an almost fair splitting by $q$-stable sets $S'_1, \dots, S'_q$. Each $S'_i$ contains at most one of the new vertices $B_j$ by $q$-stability. The sets $S_i = S'_i \setminus \bigcup_j B_j$ are still pairwise disjoint sets and $q$-stable in the original path.
    
    For $q=2$ and all $V_j$ of odd size the first part of the proof actually guarantees that $S_1$ and $S_2$ have the same cardinality. So this proves Conjecture~\ref{conj:AM} in this case. If the $V_j$ have arbitrary size we can still reduce to the case that $|V_j| = 2k_j-1$ for all~$j$, while maintaining that the cardinalities of $S_1$ and $S_2$ differ by at most one. As above we add a block $B_j$ (which for $q=2$ is either empty or consists of one vertex) to each $V_j$ to obtain $V'_j$ of size~${2k_j-1}$. We add these additional vertices at the end of the path. We now obtain two independent sets $S'_1$ and $S'_2$ of the same cardinality that are an almost fair splitting of the~$V'_j$. The sets satisfy $|S'_i \cap V_j| = k_j-1$. 
    
    Let $\ell_1 = |S'_1 \cap \bigcup_j B_j|$ and $\ell_2 = |S'_2 \cap \bigcup_j B_j|$, and suppose w.l.o.g.~that $\ell_1 \le \ell_2$. Define as before $S_i = S'_i \setminus \bigcup_j B_j$. This is still an almost fair splitting by independent sets. If $\ell_2 \le \ell_1+1$, we are done. Otherwise we need to remove another $\ell_2-\ell_1-1$ vertices from the set~$S_1$, without violating the condition that $S_1, S_2$ is an almost fair splitting. We do this in the following way: For every vertex in $S'_2 \cap \bigcup_j B_j$, consider the next vertex on the path (unless the vertex was the last vertex of the path). This vertex is in $\bigcup_j B_j$, and it is not in $S'_2$ by independence. It is in $S'_1$ at most $\ell_1$ times. Additionally excluding the case that one vertex in $S'_2 \cap \bigcup_j B_j$ might be the last vertex of the path, we can find $\ell_2-\ell_1-1$ instances, where a vertex of $S'_2 \cap \bigcup_j B_j$ is succeeded by a vertex $v$ that is in $\bigcup_j B_j \setminus (S'_1 \cup S'_2)$. The vertex $v$ is in $B_j$ for some~$j$. Now remove an arbitrary vertex in $S_1 \cap V_j$ from the set~$S_1$. Repeating this for all such vertices~$v$ results in a subset $\widetilde S_1 \subset S_1$. By construction the cardinalities of $\widetilde S_1$ and $S_2$ differ by at most one. Both sets are still independent. We have to check that they constitute an almost fair splitting.
    
    For those $j$ where we removed a vertex from $S_1 \cap V_j$ to obtain the set~$\widetilde S_1$, we have that $|\widetilde S_1 \cap V_j| = k_j-2$. For such a $j$ the set $B_j$ was nonempty, so $|V_j| = 2k_j -2$ and thus $\lfloor\frac{|V_j|+1}{2}\rfloor -1=k_j-2$. By construction $|S_2 \cap V_j| = k_j-1$, and so the sets $\widetilde S_1$ and $S_2$ cover all but one vertex of~$V_j$. This shows that $\widetilde S_1, S_2$ is a balanced almost fair splitting. The case of $q$ a power of two follows from the $q=2$ case by~\cite[Prop.~1]{alishahi2017}.
\end{proof}

\begin{remark}
	Similar reasoning to that in the proof above shows if $\chi(\KG^q(n,k)_{\widetilde{q-\mathrm{stab}}}) = \Big\lceil \frac{n-q(k-1)}{q-1} \Big\rceil$ holds for all integers $k \ge 1$ and $n \ge qk$ then the stronger Conjecture~\ref{conj:AM} holds for~$q$, provided that all $V_j$ are of size $qk_j-1$ or~$qk_j$.
\end{remark}

\section{Concluding remark}
\label{sec:rem}

\begin{remark}
	Our methods do not make use of the fact that we are concerned with independent sets in a graph. Thus our results extend to the following more general setting: Let $K$ be a simplicial complex on vertex set $V_1 \sqcup V_2 \sqcup \dots \sqcup V_m$. A \emph{fair splitting by $q$ faces} consists of pairwise disjoint faces $S_1, \dots, S_q$ of $K$ that satisfy $|S_i \cap V_j| \ge \lfloor\frac{|V_j|}{q}\rfloor$ for all~$i$ and~$j$. Similarly, pairwise disjoint faces $S_1, \dots, S_q$ of $K$ are called an \emph{almost fair splitting by $q$ faces} if $|S_i \cap V_j| \ge \lfloor\frac{|V_j|+1}{q}\rfloor-1$ for all~$i$ and~$j$ and $|V_j \setminus \bigcup_i S_i| \le q-1$ for all~$j$. Theorem~\ref{thm:main-eq} immediately generalizes to: If $K$ has at most $(q-1)n+1$ vertices for some integer $n \ge m+1$ such that there is no $\mathfrak S_q$-equivariant map $K^{*q}_\Delta \longrightarrow S(W_q^{\oplus n})$. Then $K$ admits an almost fair splitting by $q$ faces.
\end{remark}

\section*{Acknowledgements}

This research was performed during the Summer Program for Undergraduate Research 2018 at Cornell University. The authors are grateful for the excellent research conditions provided by the program. The authors would like to thank the other participants of the summer program and Thomas B\aa \aa th for helpful conversations.


\end{document}